\documentclass[11pt]{article}

\usepackage{times}
\usepackage{amsthm}
\usepackage{amsmath}
\usepackage{amssymb}
\usepackage{mathrsfs}

\def\A{{\mathcal A}}

\def\N{\mathbb{N}}

\def\F{\mathcal F}
\def\H{\mathcal H}

\def\M{{\mathcal M}}

\def\S{\mathcal S}

\def\la{\langle}
\def\ra{\rangle}

\def\I{{\rm 1\kern-.26em I}}

\def\({\left(}
\def\[{\left[}
\def\){\right)}
\def\]{\right]}

\def\si{\sigma}
\def\Si{\Sigma}

\def\G{\mathcal{G}}

\def\p{\parallel}
\def\<{\langle}
\def\>{\rangle}

\def\Op{\mathfrak{Op}}

\newtheorem{Theorem}{Theorem}[section]
\newtheorem{Remark}[Theorem]{Remark}
\newtheorem{Lemma}[Theorem]{Lemma}

\newtheorem{Corollary}[Theorem]{Corollary}
\newtheorem{Proposition}[Theorem]{Proposition}
\newtheorem{Definition}[Theorem]{Definition}

\numberwithin{equation}{section}

\begin{document}


\title{Compactness Criteria for Sets and Operators\\ in the Setting of Continuous Frames}

\date{\today}

\author{M. M\u antoiu and D. Parra 
}
\date{\small}
\maketitle \vspace{-1cm}


\begin{abstract}
To a generalized tight continuous frame in a Hilbert space $\H$\,, indexed by a locally compact space $\Si$ endowed with a Radon measure, one associates \cite{FR,RU} a coorbit theory converting spaces of functions on $\Si$ in spaces of vectors comparable with $\H$\,. If the continuous frame is provided by the action of a suitable family of bounded operators on a fixed window, a symbolic calculus emerges \cite{Ma1}, assigning operators in $\H$ to functions on $\Si$\,. We give some criteria of relative compactness  for sets and for families of compact operators, involving tightness properties in terms of objects canonically associated to the frame.
Particular attention is dedicated to a magnetic version of the pseudodifferential calculus.
\end{abstract}

\bigskip

\bigskip
{\bf Address}

\medskip
Departamento de Matem\'aticas, Universidad de Chile,

Las Palmeras 3425, Casilla 653, Santiago, Chile

\emph{E-mail:} mantoiu@uchile.cl

\emph{E-mail:} parra.alejandro@gmail.com	

\bigskip
\medskip
{\bf Acknowledgements:}
The authors are supported by {\it N\'ucleo Cientifico ICM P07-027-F "Mathematical Theory of Quantum and Classical Magnetic Systems"}.
The first named author is grateful to the Erwin Schr\"odinger Institute and to the Mittag-Leffler Institute, where part of this this paper has been written.

\section{Introduction}\label{duci}

The main goal of this article is to provide compactness criteria for bounded subsets $\Omega$ of some Banach spaces $\mathcal Y$ in terms of generalized continuous frames \cite{FG,FR,RU,Ma1}. It will be convenient in this Introduction to refer to the framework of \cite{Ma1}, less general than that of \cite{FR,RU}, but having a richer mathematical structure.

In \cite{Ma1} the framework is built on a family $\{\pi(s)\!\mid\! s\in\Si\}$ of bounded operators acting in a Hilbert space $\H$\, indexed by the points of a locally compact space $\Si$ endowed with a Radon measure $\mu$\,. Under certain convenient axioms (the square integrability condition (\ref{barrymore}) is basic), one introduces and studies a map $\phi^\pi$ from $\H\times\H$ into $L^2(\Si)$\,, and a symbolic calculus $f\mapsto\Pi(f)$ sending functions on $\Si$ into operators on $\H$ (in the spirit of a pseudodifferential theory). Formally the definitions are
\begin{equation}\label{dragaoltenu}
[\phi^\pi(u,v)](s):=\<\pi(s)u,v\>\,,\quad\ u,v\in\H\,,\ s\in\Si
\end{equation}
and
\begin{equation}\label{ionescu}
\Pi(f):=\int_\Si\!d\mu(s)f(s)\pi(s)^*\,,\quad\ f\in L^2(\Si)\,,
\end{equation}
but with suitable interpretations (using Gelfand triples for instance) they can be pushed to much more general situations. 

If $\Si$ is a locally compact group and $\pi:\Si\rightarrow\mathbb B(\H)$
is a (maybe projective) unitary, strongly continuous (maybe irreducible) representation, the framework is standard \cite{FG}. The function $\phi^\pi$ is called {\it the representation coefficient map} and $\Pi$ is {\it the integrated form of this representation}. 

However, in many physically or/and mathematically motivated situations $\Si$ is not a group. Even when it is, $\pi$ is not a projective representation; the operators $\pi(s)\pi(t)$ might not be connected to $\pi(st)$ (when the later exists) in some simple way. The need of a formalism covering non-group-like situations motivated the approach in \cite{Ma1}, to which we refer for more technical details, for constructions of involutive algebras and of coorbit spaces of vectors and symbols and for relevant examples. Actually coorbit spaces of vectors have been previously defined in \cite{FR,RU} starting with a continuous frame $W:=\{w(s)\!\mid\! s\in\Si\}\subset\H$\,, following the fundamental approach of \cite{FG}. These references, besides a deep investigation, also contain many examples and motivational issues to which we send the interested reader. In such a generality, however, the symbolic calculus $\Pi$ and the connected developments of \cite{Ma1} are not available.

To get the situation treated in \cite{Ma1} one sets essentially $w(s):=\pi(s)^*w$ for some fixed nomalised vector (window) of $\H$\,.
It is fruitful to consider the partial function $u\mapsto\phi^\pi_w(u):=\phi^\pi(u,w)$ for fixed $w$ and clearly this can be generalized to an isometry
\begin{equation}
\mathcal H\ni u\mapsto\phi_W(u):=\<u,w(\cdot)\>\in L^2(\Si)
\end{equation}
for continuous frames which are not defined by families of operators. The necessary notions from \cite{FR,Ma1,RU} are briefly reviewed in section \ref{omorniella}.

Let us come back to compactness issues. Let us fix an infinite dimensional Banach space $\mathcal Y$ and a bounded subset $\Omega$ of $\mathcal Y$\,. We assume that $\mathcal Y$ is somehow defined in the setting $(\H,\Si,\pi,\phi^\pi,\Pi)$\,. Tipically it will be one of the coorbit spaces of vectors constructed in terms of the frame\,; the Hilbert space $\H$ itself is a particular but important example. To be relatively compact $\Omega$ needs extra properties beyond boundedness, and it is natural to 
search for such properties in terms of the maps $\pi,\phi^\pi$ or $\Pi$\,. The following definition (inspired by \cite{Fei}) will be convenient.

\begin{Definition}\label{rucareanu}
Assume that the Banach space $\mathcal Y$ is endowed with a structure of Banach left module over a normed algebra $\A$\,, meaning that a left module structure $\A\times\mathcal Y\ni(a,y)\mapsto a\cdot y\in\mathcal Y$ is given and the relation $\p\! a\cdot y\!\p_\mathcal Y\,\le\,\p\! a\!\p_\A\,\p\! y\!\p_\mathcal Y$ is satisfied for every $a\in\A$ and $y\in\mathcal Y$\,. Let $\A^0\subset\A$\,; we say that {\rm the bounded set $\,\Gamma\subset\mathcal Y$ is $\A^0$-tight} if for every $\epsilon>0$ there exists $a\in\A^0$ with $\sup_{y\in\Gamma}\!\p\! a\cdot y-y\!\p_\mathcal Y\,\le\epsilon$\,.
\end{Definition}

In various situations, depending on the meaning of $\cdot$ and $\p\!\cdot\!\p_\mathcal Y$\,, tightness could have a specific intrepretation (equicontinuity, uniform concentration, etc). Note that $\mathcal Y$ is naturally a Banach left module over $\mathbb B(\mathcal Y)$\,, the Banach algebra of all bounded lineal operators in $\mathcal Y$\,, so very often we choose $\A^0\subset\mathbb B(\mathcal Y)$\,. 

Most of our results will involve characterization of relative compactness of $\Omega$ in terms of its tightness with respect to a (finite) family of Banach module structures $\{\A_j\times\mathcal Y\,\mapsto\mathcal Y\}_{j\in J}$ and corresponding subsets $\{\A_j^0\subset\A_j\}_{j\in J}$\,. An occuring generalization is using for characterization tightness of the image $\Omega':=\psi(\Omega)$ of $\Omega$ into another Banach left module. 

For illustration, let us reproduce here a slightly simplified version of Theorem \ref{rebengiuk}. We ask the map $\pi^*(\cdot):=\pi(\cdot)^*:\Si\rightarrow\mathbb B(\H)$ to be strongly continuous, to satisfy $\pi^*(s_1)=1\,$ for some $s_1\in\Si$  and to verify condition (\ref{barrymore}). Note that $C_{\rm c}(\Si)$
is contained in the $C^*$-algebra $C_0(\Si)$ of complex continuous functions on $\Si$ vanishing at infinity, which acts on $L^2(\Si)$ by pointwise multiplication.

\begin{Theorem}\label{rebengiug}
A bounded subset $\Omega$ of $\,\H$ is relatively compact if and only if any one of the next equivalent conditions holds:
\begin{enumerate}
\item
For some (every) $\,w\in\H$ the family $\phi^\pi_w(\Omega)$ is $C_{\rm c}(\Si)$-tight in $L^2(\Si)$\,.
\item
The set $\Omega$ is $\Pi\[C_{\rm{c}}(\Si)\]$-tight\,; here we use the Banach module $\mathbb B(\H)\times\H\rightarrow\H$\,.
\item 
One has $\,\underset{s\rightarrow s_0}{\lim}\,\underset{u\in\Omega}{\sup}\p\!\pi^*(s)u-\pi^*(s_0)u\!\p\,=0\,$ for every $s_0\in\Si$\,.
\end{enumerate}
\end{Theorem}

Two possible generalizations can be taken into account: (a) replace $W:=\{\pi(s)^*w\!\mid\! s\in \Si\}$ by a general continuous frame and (b) replace $\H$ by a coorbit space. Both these generalizations are considered in section \ref{onella}, but only involving the characterization $1$ of relative compactness of $\Omega$ in terms of tightness of the set $\phi_W(\Omega)$\,. One obtaines an extension of the main result of \cite{DFG}, which required $\Si$ to be a locally compact group and $w(s)=\pi(s)^*w$ for some irreducible integrable unitary representation $\pi:\Si\rightarrow\mathbb B(\H)$\,. Although substantially more general, our result allows almost the same proof as in \cite{DFG}; we include this proof for convenience and because some technical details are different. 

In fact the characterizations 2 and 3, suitably modified, would also be available in coorbit spaces. However this would need many preparations from the paper \cite{Ma1} (submitted for publication) and would involve some implicit assumptions requiring a lot of exemplifications. Therefore, at least for the moment, we decided to include compactness characterization in terms of $\pi$ and $\Pi$ only for the important case of Hilbert spaces.

We are also interested in families of compact operators. Two Banach spaces $\mathcal X$ and $\mathcal Y$ being given, the problem of deciding {\it when a set $\mathscr K$ of compact operators $:\mathcal X\rightarrow\mathcal Y$ is relatively compact in the operator norm topology} is already a classical one; for more details and motivations cf. \cite{An,GF,My,Pa,SPD} and references therein. Clearly, compactness results for subsets of $\mathcal Y$ (as those given in sections \ref{onella} and \ref{poponel}) are crucial, but extra refinaments are needed: For $\mathscr K$ to be a relatively compact set of compact operators, it is necessary but not sufficient that $\{Sx\!\mid\,\p\! x\!\p_\mathcal X \le 1,S\in\mathscr K\}$ be relatively compact in $\mathcal Y$\,; this even happens in Hilbert spaces. We discuss this problem in section \ref{onecu}; of course, if $\mathscr K:=\{S\}$ is a singleton, one gets easily criteria for the operator $S$ to be compact.

In a final Section we treat what we think to be an important example, {\it the magnetic Weyl calculus} \cite{Ma1,IMP}, which describes the quantization of a particle moving in $\mathbb R^n$ under the action of a variable magnetic field $B$\ (a closed $2$-form on $\mathbb R^n$)\,. It is a physically motivated extension of the usual pseudodifferential theory in Weyl form, which can be recovered for $B=0$\,. One reason for including this here is that it definitely stays outside the realm of projective group representations and the results on compactness existing in the literature do not apply. But it is a rather simple particular case of the formalism developped in \cite{Ma1} and the compactness criteria of the present paper work very well. We decided to present only the Hilbert space theory, having in view certain applications to the spectral theory of magnetic quantum Hamiltonians that will hopefully addressed in the future. A second reason to treat the magnetic Weyl calculus here is that it presents extra mathematical structure which has important physical implications and which also enlarges the realm of compacness criteria. If the magnetic field is zero, part of our Theorem \ref{arsinel} reproduces the classical Riesz-Kolmogorov Theorem (cf. \cite{DFG,Fei,GI} for useful discussions). Extensions of this classical result can be found in \cite{GI} and especially in \cite{Fei}; since these references use essentially the group-theoretic framework, they cannot be applied to our section \ref{one}. It would be interesting to generalize the double module formalism of \cite{Fei} to cover at least the magnetic Weyl calculus and its generalization to nilpotent Lie groups \cite{Pe,BB1,BB2}.

\section{Coorbit spaces and quantization rules associated to continuous frames}\label{omorniella}

We start with some {\it notations and conventions}:

We denote by $\overline\H$ the conjugate of the (complex separable) Hilbert space $\H$\,; it coincides with $\H$ as an additive group but it is endowed with the scalar multiplication $\alpha\cdot u:=\overline\alpha u$ and the scalar product $\<u,v\>':=\overline{\<u,v\>}$\,. If $u,v\in\H$ the rank one operator $\lambda_{u,v}\equiv\<\cdot,v\>\<u|$ is given by $\lambda_{u,v}(w):=\<w,v\>u$\,.

Let $\Si$ be a Hausdorff locally compact and $\si$-compact space endowed with a fixed Radon measure $\mu$\,. By $C(\Si)$ one denotes the space of all continuous functions on $\Si$\,, containing the $C^*$-algebra ${\rm BC}(\Si)$ composed of bounded continuous functions. The closure in ${\rm BC}(\Si)$ of the space $C_{\rm{c}}(\Si)$ of continuous compactly supported complex functions on $\Si$ is the $C^*$-algebra $C_0(\Si)$ of continuous functions vanishing at infinity.  The Lebesgue space $L^2(\Si;\mu)\equiv L^2(\Si)$ will also be used, with scalar product $\<u,v\>_{L^2(\Si)}=:\<u,v\>_{(\Si)}$\,.

For Banach spaces $\mathcal X,\mathcal Y$ we set $\mathbb B(\mathcal X,\mathcal Y)$ for the space of linear continuous operators from $\mathcal X$ to $\mathcal Y$ and use the abbreviation $\mathbb B(\mathcal X):=\mathbb B(\mathcal X,\mathcal X)$\,. The particular case $\mathcal X':=\mathbb B(\mathcal X,\mathbb C)$ refers to the topological dual of $\mathcal X$\,. By $\mathbb K(\mathcal X,\mathcal Y)$ we denote the compact operators from $\mathcal X$ to $\mathcal Y$. If $\H$ is a Hilbert space, $\mathbb B_2(\mathcal H)$ is the two-sided $^*$-ideal of all Hilbert-Schmidt operators in $\mathbb B(\H)$\,; it is a Hilbert space with the scalar product $\<S,T\>_{\mathbb B_2(\mathcal H)}:={\rm Tr}(ST^*)\,$.

\medskip
We recall now the concept of tight continuous frame and the construction of coorbit spaces, slightly modifying the approach of \cite{FR,RU}. Let us fix a family $W:=\{w(s)\!\mid\!s\in\Si\}\subset\H$ that is a tight continuous frame; the constant of the frame is assumed to be $1$ by normalizing the measure $\mu$\,. This means that the map $s\mapsto w(s)$ is assumed weakly continuous and for every $u,v\in\H$ one has
\begin{equation}\label{yak}
\<u,v\>=\int_\Si\!d\mu(s)\<u,w(s)\ra\la w(s),v\>\,.
\end{equation}
Clearly $W$  is total in $\H$ and defines an isometric operator
\begin{equation}\label{sambur}
\phi_{W}:\H\rightarrow L^2(\Si)\,,\quad\ \[\phi_{W}(u)\](s):=\<u,w(s)\>
\end{equation}
with adjoint $\phi_{W}^\dagger:L^2(\Si)\rightarrow\H$ given (in weak sense) by 
\begin{equation}\label{bivol}
\phi_{W}^\dagger(f)=\int_\Si\!d\mu(s) f(s)w(s)\,.
\end{equation}
The (Gramian) kernel associated to the frame is the function $p_W:\Si\times\Si\rightarrow\mathbb C$ given by
\begin{equation}\label{bour}
p_W(s,t):=\<w(t),w(s)\>=\[\phi_W(w(t))\](s)=\overline{\[\phi_W(w(s))\](t)}\,,
\end{equation}
defining a self-adjoint integral operator $P_W=\mathfrak{Int}(p_W)$ in $L^2(\Si)$\,.
One checks easily that $P_W=\phi_{W}\phi_{W}^\dagger$
is the final projection of the isometry $\phi_{W}$, so $P_W\!\[L^2(\Si)\]$ is a closed subspace of $L^2(\Si)$\,.
Since $\phi_{W}^\dagger\phi_{W}=1$\,, one has the inversion formula
\begin{equation}\label{bufal}
u=\int_\Si\!d\mu(t)\[\phi_{W}(u)\](t)\,w(t)\,,
\end{equation}
leading to the reproducing formula $\phi_{W}(u)=P_W\[\phi_{W}(u)\]$\,, i.e.
\begin{equation}\label{taur}
\[\phi_{W}(u)\](s)=\int_\Si\!d\mu(t)\<w(t),w(s)\>\[\phi_{W}(u)\](t)\,.
\end{equation}
Thus $\mathscr P_W(\Si):=P_W\!\[L^2(\Si)\]$ is a reproducing space with reproducing kernel $p_W$; it is composed of continuous functions on $\Si$\,.

To extend the setting above beyond the $L^2$-theory, one can supply an extra space of ``test vectors'', denoted by $\G$\,, assumed to be a Fr\' echet space  continuously and densely embedded in $\H$\,. Applying the Riesz isomorphism we are led to a Gelfand triple $(\G,\H,\G'_\si)$\,. The index $\si$ refers to the fact that on the topological dual $\G'$ we consider usually the weak-$^*$ topology. In certain circumstances one takes $\G$ to be a Banach space and sometimes it can even be fabricated from the frame $W$ and from some extra ingredients, as in Remark \ref{rimarc} below. But very often (think of the Schwartz space) the auxiliar space $\G$ is only Fr\'echet.

We shall suppose that {\it the family $W$ is contained and total in $\G$ and that $\Si\ni s\mapsto w(s)\in\G$ is a weakly continuous function.}
Then we extend $\phi_W$ to $\G'$ by $\[\phi_W(u)\](s):=\<u,w(s)\>$\,, where the r.h.s. denotes now the number obtained by applying $u\in\G'$ to $w(s)\in\G$ and depends continuously on $s$\,. By the totality of the family $W$ in $\G$\,, this extension is injective. In addition, $\Phi_W:\G'\rightarrow C(\Si)$ is continuous if one consider on $\G'$ the weak-$^*$ topology and on $C(\Si)$ the topology of pointwise convergence.

As in \cite{FG,FR,RU} and many other references treating coorbit spaces, one uses $\phi_W(\cdot)$ to pull back subspaces of functions on $\Si$\,. So let $(\M,\p\!\cdot\!\p_\M)$ be a normed space of functions on $\Si$ (more assumptions on $\M$ will be imposed when necessary) and set 
\begin{equation}\label{manole}
{\sf co}_W(\M)\equiv{\sf co}(\M):=\{u\in\G'\!\mid\! \phi_W(u)\in\M\}\,,\quad\ \p\! u\!\p_{{\sf co}(\M)}\,:=\,\p\!\phi_W(u)\!\p_\M\,.
\end{equation}
Recalling the totality of the family $W$ in $\G$\,, one gets a normed space $\({\sf co}(\M),\p\! \cdot\!\p_{{\sf co}(\M)}\right)$ and $\phi_W:{\sf co}(\M)\rightarrow\M$ is an isometry. Without extra assumptions, even when $\M$ is a Banach space, ${\sf co}(\M)$ might not be complete, so we define $\widetilde{{\sf co}}(\M)$ to be the completion. The canonical (isometric) extension of $\phi_W$ to a mapping $:\widetilde{{\sf co}}(\M)\rightarrow\M$ will also be denoted by $\phi_W$\,.
If the norm topology of ${\sf co}(\M)$ happens to be stronger than the weak-$^*$ topology on $\G'$\,, then canonically $\widetilde{{\sf co}}(\M)\hookrightarrow\G'_\si$\,.

\begin{Remark}\label{rimarc}
{\rm Following the approach of \cite{FG,FR,RU}, we indicate now a possible choice for $\G$ adapted to a given frame $W$ in $\H$\,.
Let us consider a continuous {\it admissible weight} $\alpha:\Si\times\Si\rightarrow[1,\infty)$ which is
bounded along the diagonal ($\alpha(s,s)\le C<\infty$ for all $s\in\Si\,$), symmetric ($\alpha(s,t)=\alpha(t,s)$ for all $s,t\in\Si\,$) and
satisfies $\alpha(s,t)\le \alpha(s,r)\alpha(r,t)$ for all $r,s,t\in\Si\,$. It is easy to see that
\begin{equation}\label{koala}
\mathscr A_\alpha:=\{K:\Si\times\Si\rightarrow\mathbb C\ {\rm measurable}\mid\ \p\!K\!\p_{\mathscr A_\alpha}<\infty\}
\end{equation}
is a Banach $^*$-algebra of kernels with the norm
\begin{equation}\label{coropisnita}
\p\!K\!\p_{\mathscr A_\alpha}:=\max\left\{\underset{s\in\Si}{{\rm ess}\sup}\!\int_\Si d\mu(t)|(\alpha K)(s,t)|\,,\,
\underset{t\in\Si}{{\rm ess}\sup}\!\int_\Si d\mu(s)|(\alpha K)(s,t)|\right\}\,.
\end{equation}
Picking some (inessential) point $r\in\Si$ one defines the weight $a\equiv\,a_r:\Si\rightarrow[1,\infty)$ by \ $a(s):=\alpha(s,r)$\,.
We require that the kernel $p_{W}$ given by (\ref{bour}) be an element of $\mathscr A_\alpha$\,; Then it follows that $P_W$ defines a bounded operator in the weighted Lebesgue space $L^1_a(\Si)$\,.
Then set $\G\equiv\G_{a,W}:=\{v\in\H\mid \phi_{W}(v)\in L^1_a(\Si)\}$ with the obvious norm
\begin{equation}\label{urechelnita}
\p\!v\!\p_{\G_{a,W}}:=\,\p\!\phi_{W}(v)\!\p_{L^1_a(\Si)}\,=\!\int_{\Si}\!d\mu(s)\,a(s)\,|\[\phi_{W}(v)\](s)|\,.
\end{equation}
The space $\G_{a,W}$ is a Banach space continuously and densely embedded in $\H$\,. 
In this framework, coorbit spaces were defined and thoroughly investigated in \cite{FR,RU}; if $\M$ is a Banach space then ${\sf co}(\M)$ is automatically complete. The dependence of these coorbit spaces on the frame $W$ is also studied in \cite{FR,RU}; we are going to assume that the frame $W$ is fixed.
}
\end{Remark}

Following \cite{Ma1}, we reconsider a particular case of the formalism described above. This particular case has extra structure allowing to develop a symbolic calculus and to define and study corresponding coorbit spaces of functions or "distributions" on $\Si$\,; we shall only  indicate the facts that are useful for the present paper.

Let $\pi:\Sigma\rightarrow\mathbb B(\H)$ be a map such that for every $u,v\in\H$ one has
\begin{equation}\label{barrymore}
\int_{\Sigma}\!d\mu(s)\,|\<\pi(s)u,v\>|^2=\,\p\! u\!\p^2\,\p\! v\!\p^2\,.
\end{equation}
We set $\pi(s)u=:\pi_u(s)$ and $\pi(s)^*u\equiv\pi^*(s)u=:\pi_u^*(s)$ for every $s\in\Si$ and $u\in\H$\,, getting families of functions $\{\pi_u:\Si\rightarrow\H\mid u\in\H\}$ and $\{\pi_u^*:\Si\rightarrow\H\mid u\in\H\}$\,. One also requires $\pi_u^*$ to be continuous for every $u$\,. 

The map $\Phi^\pi:\H\widehat\otimes\overline\H\rightarrow L^2(\Si)$ uniquely defined by
$$
[\Phi^\pi(u\otimes v)](s)\equiv[\phi^\pi(u,v)](s):=\<\pi(s)u,v\>
$$
is isometric, by (\ref{barrymore}). Although this was not not needed in \cite{Ma1}, we also require $\Phi^\pi$ to be surjective. For every normalized vector $w\in\H$ the map $\phi^\pi_w:\H\rightarrow L^2(\Si)$ given by $\phi^\pi_w(u):=\phi^\pi(u,w)$ is isometric. 
{\it Fixing $w$, it is clear that we are in the above framework with the tight continuous frame defined by}
\begin{equation}\label{mor}
W\equiv W(\pi,w)=\{w(s):=\pi(s)^*w\!\mid\! s\in\Si\}\,.
\end{equation}
Using existing notations one can write $\phi_W=\phi^\pi_w$ and $w(\cdot)=\pi^*_w(\cdot)$\,. After introducing a Fr\'echet space $\G$ continuously embedded in $\H$\,, one can define coorbit spaces ${\sf co}^\pi_w(\M):=\{u\in\G'\!\mid\! \phi^\pi_w(u)\in\M\}$ as it was done above. But we are not going to need them.

To define the symbolic calculus $\Pi$\,, sending functions on $\Si$ into bounded linear operators on $\H$\,, we make use of the rank one operators $\Lambda(u\otimes v)\equiv\lambda_{u,v}:=\<\cdot,v\>u$ indexed by $u,v\in\H$\,. This defines
both a map $\lambda:\H\times\H\rightarrow\mathbb F(\H)$ with values in the ideal of finite-rank operators and a unitary map $\Lambda:\H\widehat\otimes\overline\H\rightarrow\mathbb B_2(\H)$ from the Hilbert tensor product to the Hilbert space of all Hilbert-Schmidt operators on $\H$\,. Consequently $\Pi:=\Lambda\circ\(\Phi^\pi\right)^{-1}:L^2(\Si)\rightarrow \mathbb B_2(\mathcal H)$ will also be unitary; its action is uniquely defined by $\Pi[(\phi(u,v)]=\<\cdot,v\>u$\,.
Also recall \cite[Prop. 2.3]{Ma1} the formula valid in weak sense
\begin{equation}\label{nconstantin}
\Pi(f)=\int_\Si d\mu(s)f(s)\pi^*(s)\,.
\end{equation}

\section{Compactness in coorbit spaces associated to continuous frames}\label{onella}

Let us fix a tight continuous frame $W:=\{w(s)\!\mid\!s\in\Si\}$ contained and total in a Fr\'echet space $\G$ that is continuously embedded in the Hilbert space $\H$\,. It is assumed that $s\mapsto\<u,w(s)\>$ is continuous for every $u\in\G'$\,. For any normed space $\M$ of functions on  $\Si$ we have defined the coorbit space ${\sf co}_W(\mathcal M)\equiv{\sf co}(\mathcal M)$ with completion $\widetilde{{\sf co}}(\M)$\,, which will be supposed continuously embedded in $\G'_\si$\,. 

One considers a bounded subset $\Omega$ of $\widetilde{{\sf co}}(\mathcal M)$ and investigate when this subset is relatively compact in terms of the canonical mapping $\phi_W\equiv\phi$\,.  We are guided by \cite[Th. 4]{DFG}, but some preparations are needed due to our general setting.
The next abstract Lemma will be applied to $\mathcal Y=\widetilde{{\sf co}}(\M)\hookrightarrow\G'_\si$\,.

\begin{Lemma}\label{aniston}
Let $\mathfrak S(\G)$ a family of seminorms defining the topology of $\,\G$.
Assume that $\mathcal Y$ is a normed space continuously embedded in $\G'_\si$ and let $\Omega\subset\mathcal Y$ be bounded. 
\begin{enumerate}
\item
For every $p\in\mathfrak S(\G)$ there exists a positive constant $D_p$ such that
$$
|\<u,v\>|\le D_p\!\p\! u\!\p_\mathcal Y p(v)\,,\quad\ \forall\,v\in\G,\,u\in\mathcal Y\,.
$$
\item
Seen as a subset of $\,\G'$\,, the set $\Omega$ is equicontinuous and (consequently) relatively compact in the weak-$^*$ topology.
\end{enumerate}
\end{Lemma}

\begin{proof}
1 is standard; actually the condition is equivalent to $\mathcal Y\hookrightarrow\G'_\si$\,.

\medskip
\noindent
2. A base of neighborhoods of the origin in $\G$ is
$$
\big\{U(p\,;\delta):=\{v\in\G\!\mid\! p(v)<\delta\}\mid p\in\mathfrak S(\G),\delta>0\big\}\,.
$$
Assume that $\p\!u\!\p_\mathcal Y\,\le\! M$ for every $u\in\Omega$\,. Let $\epsilon>0$\,and $p\in\mathfrak S(\G)$\,.
Using 1, for every $v\in U\!\(p\,;\frac{\epsilon}{MD_p}\right)$ and every $u\in\Omega$ one gets
$$
|\<u,v\>|\le D_p\!\p\!u\!\p_\mathcal Y p(v)\le D_p M p(v)\le\epsilon\,,
$$
and this is equicontinuity. The statement concerning relative compactness follows from the Bourbaki-Alaoglu Theorem \cite{Ja}.
\end{proof}

Let us denote by $\mathcal K(\Si)$ the family of characteristic functions of all compact subsets in $\Si$\,. It can be seen as a subset of the normed algebra $L^\infty_{{\rm c}}(\Si)$ formed of $L^\infty$ functions on $\Si$ which are essentially compactly supported.

We assume that $\M$ {\it is a solid Banach space of functions with absolutely continuous norm} (cf. \cite{BS}; see also \cite{DFG}). We recall that such a space contains all the characteristic functions of sets $M\subset\Si$ with $\mu(M)<\infty$ and given $f,g:\Sigma\to\mathbb C$ two $\mu$-measurable functions, if $|{f(s)}|\leq|{g(s)}|$ almost everywhere and $g \in \mathcal{M}$  then $f\in \mathcal{M}$ and $\p\!f\!\p_\M\,\leq\,\p\!g\!\p_\M$\,. It follows that $\mathcal M$ is a Banach $L_{{\rm c}}^\infty(\Si)$-module. In addition, for all $f,g \in \M$ the following dominated convergence theorem holds: whenever $f_n:\Sigma\to\mathbb C$ are measurable, $|f_n|\leq|g|$ and $f_n \to f$ $\mu$-a.e. then $\p\!f_n-f\!\p_{\M}\,\to0$\,. Any such space is reflexive \cite[Ch. 1, Prop. 3.6 \& Th. 4.1]{BS}.

\begin{Theorem}\label{damon}
Let us assume that $\mathcal M$ is a solid Banach space of functions on $\Si$ with absolutely continuous norm. Then the bounded subset $\,\Omega$ of $\,\widetilde{{\sf co}}(\mathcal M)$ is relatively compact if and only if $\phi(\Omega)$ is $\mathcal K(\Si)$-tight in $\mathcal M$\,.
\end{Theorem}

\begin{proof}
We start with the {\it only if part}. 
By relative compactness of $\Omega$\,, for any $\epsilon>0$ there is a finite subset $F$ such that
$$
\min_{v\in F}\p\! u-v\!\p_{\widetilde{{\sf co}}(\M)}\,\le\frac{\epsilon}{2}\,,\quad\ \forall\,u\in \Omega\,.
$$
Recalling that $\Si$ has been assumed $\si$-compact, there is an increasing family $\{L_m\!\mid\! m\in\N\}$ of compact subsets of $\Si$ with $\cup_m L_m=\Si$\,. Since pointwisely $|\chi_{L_m}\phi(v)|\leq|\phi(v)|$ and $\chi_{L_m}\phi(v)\xrightarrow{m\to\infty}\phi(v)$\,, there is a compact set $L\subset\Si$ with complement $L^c$ such that
$$
\max_{v\in F}\p\!\chi_{L^c}\phi(v)\!\p_\M\,\leq\frac\epsilon2\,.
$$
Then, for every $u\in\Omega$\,, using the information above and the fact that $\phi:\widetilde{{\sf co}}(\mathcal M)\rightarrow\M$ is isometric,
\begin{align*}
 \p\!\chi_{L^c}\phi(u)\!\p_\M\,\leq&\,\min_{v\in F}\(\p\!\chi_{L^c}\phi_w(u-v)\!\p_\M+\p\!\chi_{L^c}\phi(v)\!\p_\M\) \\
 \leq&\,\min_{v\in F}\p\!\phi(u-v)\!\p_\M+\frac\epsilon2 \\
 =&\,\min_{v\in F}\p\!u-v\!\p_{\widetilde{{\sf co}}(\M)}+\frac\epsilon2\le\epsilon\,.
\end{align*}

We now prove {\it the converse}. Knowing that $\phi(\Omega)$ is $\mathcal K(\Si)$-tight in $\mathcal M$\,, one needs to show that every sequence $(u_n)_{n\in\N}\subset\Omega$ has a convergent subsequence. 
By Lemma \ref{aniston} the bounded set $\Omega\subset\widetilde{{\sf co}}(\mathcal M)$ is relatively compact in $\G'_\si$\,, so $(u_n)_{n\in\N}$ has a $^*$-weakly convergent subsequence $u_j\rightarrow u_\infty\in\G'$\,:
\begin{equation}\label{galabru}
\<u_j,v\>\rightarrow\<u_\infty,v\>\ \ \ {\rm for\ any}\,\ v\in\G\,.
\end{equation}
Putting $v:=w(s)$ in (\ref{galabru}), we get for every $s\in\Si$
$$
\<u_j,w(s)\>=\[\phi(u_j)\]\!(s)\rightarrow\[\phi(u_\infty)\]\!(s)=\<u_\infty,w(s)\>\,.
$$
Therefore the sequence $\(\phi(u_j)\right)_{j\in\N}$ is pointwise Cauchy. We shall convert this in the norm convergence 
\begin{equation}\label{defunes}
\p\!\phi(u_j)-\phi(u_k)\!\p_\mathcal M\,\rightarrow 0\ \ {\rm when}\ \ j,k\rightarrow\infty\,.
\end{equation}
Then the proof would be finished since $\phi:\widetilde{{\sf co}}(\mathcal M)\rightarrow\mathcal M$ is isometric: $(u_j)_{j\in\N}$ will be Cauchy in $\widetilde{{\sf co}}(\mathcal M)$\,, thus convergent (to $u_\infty$ of course). 

By tightness, pick a compact subset $L\subset\Si$ such that $\p\!\chi_{L^c}\phi(u)\!\p_{\mathcal M}\,\le\epsilon$ for every $u\in\Omega$\,; then we get
\begin{equation}\label{popa}
\p\!\chi_{L^c}\phi(u_j-u_k)\!\p_{\mathcal M}\,\le 2\epsilon\,,\quad\ \forall\,j,k\in\N\,.
\end{equation}
Since $\widetilde{{\sf co}}(\mathcal M)$ is continuously embedded in $\G'_\si$\,, for any seminorm
$p\in\mathfrak S(\G)$ there exist positive constants $D_p,D_p'$ such that for every $s\in\Si$  
$$
\sup_{j,k}|\<u_j-u_k,w(s)\>|\le D_p\sup_{j,k}\p\! u_j-u_k\!\p_{\widetilde{{\sf co}}(\mathcal M)}p[w(s)]\le D'_p\,p[w(s)]\,.
$$
By our assumption on $W$ and by the Uniform Boundedness Principle the family $\{w(s)\!\mid\!s\in L\}$ is bounded in $\G$\,, so we get 
$$
|\[\phi(u_j-u_k)\](s)|\le D'_{p} C_{p,L}\,,\quad \forall\,j,k\in\N,\,s\in L\,.
$$
Anyhow we obtain by the Dominated Convergence Theorem
\begin{equation}\label{morgenstern}
\p\!\chi_L\phi(u_j-u_k)\!\p_{\mathcal M}\,\rightarrow 0\ \ {\rm when}\ \ j,k\rightarrow\infty\,.
\end{equation}
Putting (\ref{morgenstern}) and (\ref{popa}) together one gets (\ref{defunes}) and thus the result.
\end{proof}

\begin{Remark}\label{iurie}
{\rm Let $S$ be an bounded operator from the Banach space $\mathcal X$ to $\widetilde{{\sf co}}(\M)$\,. Then $S$ is a compact operator if and only if for every $\epsilon\!>\!0$ there exists a compact set $L\subset\Si$ such that 
\begin{equation}\label{dar}
\p\!\chi_{L^c}\!\circ\!\phi_W\circ S\!\p_{\mathbb B(\mathcal X,\M)}\,\leq\epsilon\,.
\end{equation}
This follows easily applying Theorem \ref{damon} to the set $\Omega:=S\!\(\mathcal X_{\[1\]}\right)$ and using the explicit form of the operator norm\,. Here $\mathcal X_{[1]}$ denotes the closed unit ball in the Banach space $\mathcal X$\,.}
\end{Remark}

\section{Compactness in Hilbert spaces}\label{poponel}

To have an ampler setting, we turn now to the particular case described in the last part of Section \ref{omorniella}. Thus a family $\{\pi(s)\!\mid\! s\in\Si\}$ of bounded operators in the Hilbert space $\H$ is given. We recall that $s\mapsto \pi(s)^*\in\mathbb B(\H)$ is strongly continuous and that (\ref{barrymore}) is verified
for every $u,v\in\H$\,. Then $\phi^\pi_w:\H\rightarrow L^2(\Si)$ defined by $\[\phi^\pi_w(u)\](s):=\<\pi(s)u,w\>$ is well-defined and isometric for every normalized vector $w$ of the Hilbert space $\H$.

\begin{Theorem}\label{rebengiuk}
Let $\Omega$ be a bounded subset of $\,\H\,$. Consider the following assertions:
\begin{enumerate}
\item
$\Omega$ is relatively compact.
\item
For every $\,w\in\H$ the family $\phi^\pi_w(\Omega)$ is $\mathcal K(\Si)$-tight in $L^2(\Si)$\,.
\item
There exists $\,w_0\in\H$ such that the family $\phi^\pi_{w_0}(\Omega)$ is $\mathcal K(\Si)$-tight in $L^2(\Si)$\,.
\item
For each $\epsilon>0$ there exists $f\in C_{\rm{c}}(\Si)$ with $\,\underset{u\in\Omega}{\sup}\!\p\!\Pi(f)u-u\!\p\,\le\epsilon$ 
(i.e. $\Omega$ is $\Pi\[C_{{\rm c}}(\Si)\]$-tight)\,.
\item 
One has $\,\underset{s\rightarrow s_0}{\lim}\,\underset{u\in\Omega}{\sup}\p\!\pi(s)^*u-\pi(s_0)^*u\!\p\,=0\,$ for every $s_0\in\Si$\,.
\item
For every $\epsilon>0$ and $s_0\in\Si$ there exists $g\in C_{\rm{c}}(\Si)$ such that $\,\underset{u\in\Omega}{\sup}\!\p\!\Pi(g)u-\pi(s_0)^*u\!\p\,\le\epsilon$\,.
\end{enumerate}
Then 1,\,2,\,3 and 4 are equivalent, they imply 5, which in its turn implies 6\,. 
Thus, if we assume that $\pi(s_1)^*=1\,$ for some $s_1\in\Si$\,, then all the six assertions are equivalent.
\end{Theorem}

\begin{proof}
The equivalence of the points 1,\,2 and 3 follows from Theorem \ref{damon}\,, since in this case $\H={\sf co}\!\[L^2(\Si)\]$ and $\M:=L^2(\Si)$ is indeed a solid Banach space of functions with absolutely continuous norm.

\medskip
\noindent
$1\Rightarrow 4$.
Let $\Omega\subset\H$ be relatively compact and, for some $\epsilon>0$\,, let $F$ be a finite subset such that for each $u\in\Omega$ there exists $v_u\in F$ with $\p\! u-v_u\!\p\,\le\epsilon/4$\,. The subspace $\F$ generated by $F$ will be finite-dimensional and thus the corresponding projection $P$ will be a finite-rank operator satisfying $P v=v$ for every $v\in F\,$. Then for every $u\in\Omega$
\begin{equation}\label{bourvi}
\begin{aligned}
\p\!Pu-u\!\p\,&\le\,\p\!Pu-Pv_u\!\p+\p\!Pv_u-v_u\!\p+\p\!v_u-u\!\p\,\le 2\p\!u-v_u\!\p\,\le\epsilon/2\,.
\end{aligned}
\end{equation}
Notice that $\{\,\Pi(f)\!\mid\!f\in C_{\rm{c}}(\Si)\,\}$ is a dense set of compact operators. To see this, use the fact that $\Pi:L^2(\Si)\rightarrow\mathbb B_2(\H)$ is an isometric isomorphism and that $C_{{\rm c}}(\Si)$ is dense in $L^2(\Si)$\,; the topology of $\mathbb B_2(\H)$ is stronger than that of $\mathbb B(\H)$, while $\mathbb K(\Si)$ is the closure of $\mathbb B_2(\H)$ in the operator norm.
Let now $M:=\sup_{u\in\Omega}\!\p\!u\!\p$\,; by density there is some $f\in C_{{\rm c}}(\Si)$ with $\p\!P-\Pi(f)\!\p_{\mathbb B(\mathcal H)}\,\le\epsilon/2M$\,. From this and from (\ref{bourvi}) the conclusion follows immediately.

\medskip
\noindent
$4\Rightarrow 1$.
To prove the converse, for $\epsilon>0$ choose $f\in C_{\rm{c}}(\Si)$ such that $\,\sup_{u\in\Omega}\!\p\!\Pi(f)u-u\!\p\,\le\epsilon/2$\,.
Since $\Pi(f)$ is a compact operator and $\Omega$ is bounded, the range $\Pi(f)\Omega$ is relatively compact, so there is a finite set $G$ such that for each $u\in\Omega$ there is an element $v^u\in G$ with $\p\! \Pi(f)u-v^u\!\p\,\le\epsilon/2$\,. Then for $u\in\Omega$ one has
$$
\p\! u-v^u\!\p\,\le\,\p\!u-\Pi(f)u\!\p+\p\!\Pi(f)u-v^u\!\p\,\le\,\epsilon/2+\epsilon/2=\epsilon\,,
$$
so the set $\Omega$ is totally bounded.

\medskip
$4\Rightarrow 5$. Setting $S^\perp:=1-S$\,, we compute for $s_0\in\Si\,,\,u\in\Omega\,,\,f\in C_{\rm{c}}(\Si)$ and $s$ belonging to a neighborhood $V$ of $s_0$\,:
$$
\begin{aligned}
\p\!\pi(s)^*u-\pi(s_0)^*u\!\p\,&\le\,\p\![\pi(s)^*-\pi(s_0)^*]\,\Pi(f)u\!\p+\p\![\pi(s)^*-\pi(s_0)^*]\,\Pi(f)^\perp u\!\p\\
&\le\,\sup_{u\in\Omega}\p\! u\!\p\,\p\![\pi(s)^*-\pi(s_0)^*]\,\Pi(f)\!\p_{\mathbb B(\H)}+
\,2\sup_{t\in V}\p\!\pi(t)^*\!\p_{\mathbb B(\H)}\sup_{u\in\Omega}\p\!\Pi(f)^\perp u\!\p.
\end{aligned}
$$
The first term is small for $s$ belonging to a suitable neighborhood $V$, because $\Omega$ is bounded, $\pi^*$ is strongly continuous and this is improved to norm continuity by multiplication with the compact operator $\,\Pi(f)$\,. The second term is also small for some suitable $f$\,, because of the assumption $4$ and since $\p\!\pi^*(\cdot)\!\p_{\mathbb B(\H)}$ is bounded on the compact set $\overline V$ (use the Uniform Boundedness Principle and the strong continuity of $\pi^*$)\,. 

\medskip
$5\Rightarrow 6$. Compute for any positive $g\in C_{\rm{c}}(\Si)$ with $\int_\Si\!gd\mu=1$
$$
\begin{aligned}
\p\!\Pi(g)u-\pi(s_0)^*u\!\p\,=\Big\Vert\int_\Si\!d\mu(s)g(s)[\pi(s)^*u-\pi(s_0)^*u]\,
\Big\Vert\le\int_\Si\!d\mu(s)g(s)\!\p\![\pi(s)^*-\pi(s_0)^*]\,u\!\p
\end{aligned}
$$
and then use $5$ and require $g$ to have support inside the convenient neighborhood of the point $s_0$\,.
\end{proof}

\begin{Remark}\label{siedentop}
{\rm Among the possible applications of Theorem \ref{rebengiuk}, let us mention one concerning the connection between spectral and dynamical properties of self-adjoint operators. So let $H$ be a (maybe unbounded) self-adjoint operator in the Hilbert space $\H$\,. We denote by $\{e^{itH}\!\mid\! t\in\mathbb R\}$ the evolution group generated by $H$ (a $1$-parameter strongly continuous group of unitary operators) and for each $u\in\H$ we set $[u]^H$ for the quasiorbit of $u$ under this group, i.e. $[u]^H$ is the norm-closure of the orbit $\{e^{itH}u\!\mid\! t\in\mathbb R\}$\,. By $\H_{{\rm p}}(H)$ we denote the closed subspace of $\H$ generated by the eigenvectors of $H$\,. It is known (see \cite{GI} for instance) that a vector $u$ belongs to $\H_{{\rm p}}(H)$ if and only if $[u]^H$ is a compact subset of $\H$\,. Applying Theorem \ref{rebengiuk} to the bounded set $\Omega:=[u]^H$ one gets various characterizations for the vector to belong to the spectral subspace $\H_{{\rm p}}(H)$ in terms of one of the objects $\pi,\Pi$ or $\phi^\pi_w$\,. This is valuable especially when $H$ is the quantum Hamiltonian of some physical system described in $\H$ and the family $\pi(\cdot)$
also has some physical meaning.
}
\end{Remark}

\medskip
{\it For simplicity, we are always going to assume that $\,\pi(s_1)^*=1\,$ for some $s_1\in\Si$\,.} Below $\H_{[1]}$ denotes the closed unit ball of the Hilbert space $\H$\,.

\begin{Corollary}\label{moraru}
Let $\mathcal X$ be a Banach space and $S\in\mathbb B(\mathcal X,\H)$\,. The next assertions are equivalent:
\begin{enumerate}
\item
$S$ is a compact operator.
\item
The set $\phi^\pi_w(S\mathcal X_{[1]})$ is $\mathcal K(\Si)$-tight in $L^2(\Si)$ for some (every) $w\in\H$\,.

Writting $M_{\chi_L}^\perp$ for the operator of multiplication by the function $1-\chi_L$ in $L^2(\Si)$\,, this can be restated: for every $\epsilon>0$ there is a compact subset $L$ of $\,\Si$ such that $\,\p\!M_{\chi_L}^\perp\circ\phi^\pi_w\circ S\!\p_{\mathbb B(\mathcal X,L^2)}\,\le\epsilon$\,.
\item
For every $\,\epsilon>0$ there is some $f\in C_{{\rm c}}(\Si)$ such that $\,\p\![\Pi(f)-1]S\!\p_{\mathbb B(\mathcal X,\H)}\,\le\epsilon$\,.
\item
The map $\,\Si\ni s\mapsto\pi(s)^*S\in\mathbb B(\mathcal X,\H)$ is norm-continuous.
\end{enumerate}
\end{Corollary}

\begin{proof}
This is a simple consequence of Theorem \ref{rebengiuk}, since $S$ is a compact operator if and only if $\Omega:=S\mathcal X_{[1]}$ is relatively compact in $\H$\,; also use $\p\!T\!\p_{\mathbb B(\mathcal X,\mathcal H)}=\underset{x\in\mathcal X_{[1]}}{\sup}\!\p\!Tx\!\p$\,.
\end{proof}

\begin{Remark}\label{olivier}
{\rm
Let us have a look at the implication $4\,\Rightarrow\,1$\,. We could say that a strongly continuous function $\rho:\Gamma\rightarrow\mathbb B(\H)$ ($\Gamma$ is a topological space)
{\it characterizes compactness} if for any $S\in\mathbb B(\mathcal X,\H)$ (and for any Banach space $\mathcal X$) the fact that the function $\rho_S(\cdot):=\rho(\cdot)S$ is
norm-continuous implies $S\in\mathbb K(\H)$\,. In particular, our function $\pi^*$ does this. Many other don't; think for instance that $\rho$ is already norm-continuous or that all the ranges $\rho(\gamma)\H$ are orthogonal on some fixed proper infinitely dimensional subspace.
}
\end{Remark}

\begin{Remark}\label{rafaila}
{\rm It is easy to interpret the point $3$ as a tightness condition, since $\mathbb B(\mathcal X,\H)$ is a left Banach module over $\mathbb B(\H)$ under operator multiplication.}
\end{Remark}

\section{Compactness in spaces of compact operators}\label{onecu}

Our next problem is to describe relative compactness of subsets $\mathscr K$ of the Banach space $\mathbb K(\mathcal X,\widetilde{{\sf co}}(\M))$ of all compact operators from an arbitrary Banach space $\mathcal X$ to (the completion of) a coorbit space. Even the case $\widetilde{{\sf co}}(\M)=\H$ is interesting.
Of course, the key fact is that we have now convenient descriptions of relative compactness in the final Banach space $\widetilde{{\sf co}}(\M)$\,. But this can be used efficiently only taking into account some rather deep abstract facts.

One could hope that the good conditions would be uniform versions of (\ref{dar}) (the same $L$ for a given $\epsilon$ and for all the elements $S$ of $\mathscr K$)
or of the conditions 2,3 or 4 in Corollary \ref{moraru}\,. Clearly such a guess is connected to the notion of {\it collective compactness}. If $\mathcal X$ and $\mathcal Y$ are Banach spaces, a subset $\mathscr L$ of $\mathbb K(\mathcal X,\mathcal Y)$ is called collectively compact if $\mathscr L\mathcal X_{[1]}:=\cup_{S\in\mathscr L}S\mathcal X_{[1]}$ is relatively compact in $\mathcal Y$\,. It can be shown that if $\mathscr L$ is relatively compact in $\mathbb K(\mathcal X,\mathcal Y)$ then it is also collectively compact. To see that the converse is false, take for simplicity $\mathcal X=\mathcal Y=\H$ a Hilbert space. It is easy to check that $\mathscr L$ is relatively compact if and only if $\mathscr L^*:=\{S^*\!\mid\! S\in\mathscr L\}$ is relatively compact. But such a stability under taking the family of adjoints fails dramatically in the case of collective compactness. Let $\{e_j\mid j\in\N\}$ be an orthonormal base in $\H$ and set $\mathscr L:=\{\<\cdot,e_j\>e_1\!\mid\! j\in\N\}$\,. Then $\mathscr L$ is collectively compact while $\mathscr L^*:=\{\<\cdot,e_1\>e_j\!\mid\! j\in\N\}$ is not!

Let us return to the Banach case and denote by $\mathcal X'$ and $\mathcal Y'$\,, respectively, the topological duals of the spaces $\mathcal X$ and $\mathcal Y$\,.
It has been considered a success proving finally \cite{Pa,An} that {\it $\mathscr L\subset \mathbb K(\mathcal X,\mathcal Y)$ is relatively compact if and only if both $\mathscr L$ and $\mathscr L'\subset\mathbb K(\mathcal Y',\mathcal X')$ are collectively compact}. By definition, $\mathscr L'$ is composed of the transposed operators $S':\mathcal Y'\rightarrow\mathcal X'$ with $S\in\mathscr L$\,. 

This result is not yet handy for our problem (in which $\mathcal Y=\widetilde{{\sf co}}(\M)$)\,, because in general we do not know anything about compactness of the subsets of $\mathcal X'$\,. On the other hand, much later  \cite{SPD} it has been shown that $\mathscr L'\subset\mathbb K(\mathcal Y',\mathcal X')$ is collectively compact if and only if $\mathscr L$ is {\it equicompact}, in the sense that there is a sequence $\mathcal X'\ni x'_n\rightarrow 0$ such that $\sup_{S\in\mathscr L}\!\!\p\!Sx\!\p_{\mathcal Y}\,\le\sup_n|\<x_n',x\>|$ for every $x\in\mathcal X$\,. 

\begin{Remark}\label{buster}
{\rm In \cite{My} it is also shown that if $\mathcal X$ does not contain an isomorphic copy of $l^1$ then $\mathscr L$ is relatively compact if and only if it is collectively compact and {\it uniformly weak-norm continuous} (if $x_n\rightarrow 0$ weakly then $\sup_{S\in\mathscr L}\!\p\!Sx_n\!\p_\mathcal Y\rightarrow 0$)\,. This also follows from \cite{SPD}, while \cite{GF} contains a related result.
}
\end{Remark}

Using all these, the notions introduces in section \ref{omorniella}  and Theorem \ref{damon} one gets easily

\begin{Corollary}\label{cutare}
Let us assume that $\mathcal M$ is a solid Banach space of functions on $\Si$ with absolutely continuous norm, let $\mathcal X$ be a Banach space and $\mathscr K$ a subset of $\,\mathbb B\[\mathcal X,\widetilde{{\sf co}}(\M)\]$\,. Then $\mathscr K$ is a compact family of compact operators if and only if 
\begin{enumerate}
\item
For every $\epsilon\!>\!0$ there exist a compact set $L\subset\Si$ such that
$\underset{S\in\mathscr K}{\sup}\!\p\!\chi_{L^c}\!\circ\!\phi_W\circ S\!\p_{\mathbb B(\mathcal X,\M)}\,\leq\epsilon$ 

and          
\item
There is a sequence $\mathcal X'\ni x_n'\rightarrow 0$ such that $\underset{S\in\mathscr K}{\sup}\!\p\! \phi_W(Sx)\!\p_{\mathcal M}\,\le\sup_n|\<x_n',x\>|\,$ for every $x\in\mathcal X$\,.                                                                                                                                               
\end{enumerate}
If $\mathcal X$ does not contain an isomorphic copy of $\,l^1\,$, then 2 can be replaced by

\medskip
2'. If $x_n\rightarrow 0$ weakly then $\underset{S\in\mathscr K}{\sup}\!\p\!\phi_W(Sx_n)\!\p_{\M}\,\rightarrow 0$\,.
\end{Corollary}

We refer now to the situation explored in section \ref{poponel}, recalling the objects $(\pi,\phi^\pi_w,\Pi)$\,; for simplicity we only consider the case $\mathcal X=\H$\,. There are several ways to characterize relative compactness of subsets of $\mathbb K(\H)$\,, relying on Theorem \ref{rebengiuk} and the discussion preceding Corollary \ref{cutare}. We present the one involving collective compactness of the family and of the family of adjoints and leave to the interested reader the easy task to state others, maybe also for the case of a general Banach space $\mathcal X$\,. The setting is that of section \ref{poponel}; it is also assumed that $\pi(s_1)^*=1$ for some $s_1\in\Si$\,. 

\begin{Corollary}\label{bibanu}
Let $\mathscr K$ be a family of bounded operators in $\H$\,. The following assertions are equivalent:
\begin{enumerate}
\item
$\mathscr K$ is a relatively compact family of compact operators.
\item
For some (any) $w\in\H$\ the family $\{\phi_w(S\H_{[1]})\!\mid\! S\in\mathscr K\cup\mathscr K^*\}$ is uniformly tight in $L^2(\Si)$\,. 

This condition means that for every strictly positive $\epsilon$ there exists a compact subset $L$ of $\,\Si$ such that
$\,\underset{S\in\mathscr K\cup\mathscr K^*}{\sup}\p\!M_{\chi_L}^\perp\circ\phi_w\circ S\!\p_{\mathbb B(\H,L^2)}\,\le\epsilon$\,.
\item
For every $\,\epsilon>0$ there exists $f\in C_{{\rm c}}(\Si)$ such that $\underset{S\in\mathscr K\cup\mathscr K^*}{\sup}\!\!\p\![\Pi(f)-1]S\!\p_{\mathbb B(\H)}\,\le\epsilon$
(also a tightness statement).
\item
$\{\Si\ni s\mapsto\pi(s)^*S\in\mathbb B(\H)\mid S\in\mathscr K\cup\mathscr K^*\}$ is an equicontinuous family.
\end{enumerate}
\end{Corollary}

\section{Compactness in the magnetic Weyl calculus}\label{one}

The magnetic pseudodifferential calculus \cite{MP1,IMP} has as a background the problem of quantization of a physical system consisting in a spin-less particle moving in the euclidean space $X:=\mathbb R^n$ under the influence of a magnetic field, i.e.
a closed $2$-form $B$ on $X$ ($dB=0$), given by matrix-component functions $B_{jk}=-B_{kj}:X\rightarrow\mathbb R\,,\ \ j,k=1,\dots,n$\,.
For convenience we are going to assume that the components $B_{jk}$ belong to $C^\infty_{\rm{pol}}(X)$\,,
the class of smooth functions on $X$ with polynomial bounds on all the derivatives.  The magnetic field can be written
in many ways as the differential $B=dA$ of some $1$-form $A$ on $X$ called {\it vector potential}.
One has $B=dA=dA'$ iff $A'=A+d\varphi$ for some $0$-form $\varphi$ (then they are called equivalent).
It is easy to see that the vector potential can also be chosen of class $C^\infty_{\rm{pol}}(X)$\,; this will be tacitly assumed.

One would like to develop a symbolic calculus $a\mapsto\Op^A(a)$ taking the magnetic field into account. Basic requirements are:
(i) it should reduce to the standard Weyl calculus \cite{Fo,Gr} for $A=0$ and (ii) the operators $\Op^A\!(a)$ and $\Op^{A'}\!(a)$ should
be unitarily equivalent (independently on the symbol $a$) if $A$ and $A'$ are equivalent; this is called {\it gauge covariance}
and has a fundamental physical meaning. There are many ways to justify the formulae, including geometrical or classical mechanics reasons or ideas coming from group cohomology and the theory of crossed product $C^*$-algebras. The one closest to our approach it to think of the emerging
symbolic calculus as a functional calculus for the family of non-commuting self-adjoint operators
$(Q_1,\dots,Q_n;P^A_1,\dots,P^A_n)$ in $\H:=L^2(X)$\,. Here $Q_j$ is one of the components of
the position operator, but the momentum $P_j:=-i\partial_j$ is replaced by {\it the magnetic momentum}
$P^A_j:=P_j-A_j(Q)$ where $A_j(Q)$ indicates the operator of multiplication with the function $A_j\in C^\infty_{\rm{pol}}(X)$\,. Notice the commutation relations
\begin{equation}\label{cocostarc}
i[Q_j,Q_k]=0\,,\quad i[P^A_j,Q_k]=\delta_{jk}\,,\quad i[P^A_j,P^A_k]=B_{jk}(Q)\,.
\end{equation}
Let us set $\Si:=X\times X^*$ (called {\it the phase space} and isomorphic to $\mathbb R^{2n}$)\,, on which we consider the Lebesgue measure $d\mu(x,\xi)\equiv dxd\xi$\,. One defines {\it the magnetic Weyl system}
\begin{equation}\label{termita}
\pi^A:\Si\rightarrow\mathbb B(\H)\,,\quad\ \pi^A(x,\xi):=\exp\[i\(x\cdot P^A-Q\cdot\xi\)\]
\end{equation}
and gets in terms of the circulation of the $1$-form $A$ through the segment $[y,y+x]:=\{y+tx\mid t\in[0,1]\}$ the explicit formula
\begin{equation}\label{croitor}
\[\pi^A(x,\xi)u\]\!(y)=e^{-i\(y+\frac{x}{2}\)\cdot\xi}\,\exp\[(-i)\!\underset{[y,y+x]}{\int}\!A\]\,u(y+x)\,.
\end{equation}
These operators depend strongly continuous of $(x,\xi)$ and satisfy $\pi^A(0,0)=1$ and $\pi^A(x,\xi)^*=\pi^A(x,\xi)^{-1}=\pi^A(-x,-\xi)$ (thus being unitary). {\it However they do not form a projective representation of $\,\Si=X\times X^*$.} Actually they satisfy
\begin{equation}\label{tzetze}
\pi^A(x,\xi)\,\pi^A(y,\eta)=\omega^B[(x,\xi),(y,\eta);Q]\,\pi^A(x+y,\xi+\eta)\,,
\end{equation}
where $\omega^B[(x,\xi),(y,\eta);Q]$ only depends on the $2$-form $B$ and denotes the operator of multiplication
in $L^2(X)$ by the function
\begin{equation}\label{fluture}
X\ni z\rightarrow\omega^B[(x,\xi),(y,\eta);z]:=\exp\left[\frac{i}{2}\,(y\cdot\xi-x\cdot\eta)\right]
\exp\left[(-i)\!\!\!\!\!\underset{<z,z+x,z+x+y>}{\int}\!\!\!\!B\,\right]\,.
\end{equation}
Here the distinguished factor is constructed with the flux (invariant integration) of the magnetic field
through the triangle defined by the corners $z$, $z+x$ and $z+x+y$.

A straightforward computation leads to {\it the magnetic Fourier-Wigner function}
\begin{equation*}
 \begin{aligned}
&\[\Phi^A(u\otimes v)\](x,\xi)\equiv\[\phi^A_v(u)\](x,\xi):=\<\pi^A(x,\xi)u,v\>\\
=&\int_X dy\,e^{-iy\cdot\xi}\,\exp\[(-i)\!\!\underset{[y-x/2,y+x/2]}{\int}\!\!A\]u(y+x/2)\,\overline{v(y-x/2)}.
\end{aligned}
\end{equation*}

It can be decomposed into the product of the multiplication by a function with values in the unit circle,
a change of variables with unit jacobian and a partial Fourier transform. All these are isomorphisms, so $\Phi^A:L^2(X)\widehat\otimes L^2(X)\rightarrow L^2(\Si)$ defines a unitary transformation. 
{\it Thus we get a formalism which is a particular case of the one presented at the end of section \ref{omorniella}.}
Therefore one can apply all the prescriptions and get the correspondence 
\begin{equation}\label{cutarel}
f\mapsto\Pi^A(f):=\int_{\Si}\!f(x,\xi)\,\pi^A(-x,-\xi)\,dxd\xi \,.
\end{equation}
In fact people are interested in the (symplectic) Fourier transformed version $a(Q,P^A\,)\equiv\Op^A(a):=\Pi^{A}[\mathfrak F^{-1}(a)]
$\,. The resulting {\it magnetic Weyl calculus} is given by
\begin{equation}\label{op}
\left[\mathfrak{Op}^{A}(a)u\right](x)=(2\pi)^{-n}\!\!\int_X\!\!dy\int_{X^*}\!\!\!d\xi\,\exp\left[i(x-y)\cdot\xi\right]
\exp\left[-i\int_{[x,y]}A\right]a\left(\frac{x+y}{2},\xi\right)u(y).
\end{equation}
An important property of (\ref{op}) is {\it gauge covariance}, as hinted above:
if $A'=A+d\rho$  defines the same magnetic field as $A$, then $\Op^{A'}\!(a)=e^{i\rho}\,\Op^{A}(a)\,e^{-i\rho}$.
By killing the magnetic phase factors in all the formulae above one gets the defining relations of the usual Weyl calculus.

Due to the particular structure, one can introduce $\{U^A(x):=\pi^A(x,0)\!\mid\!x\in X\}$ (generalizing the group of translations for $A\ne 0$) and $\{V(\xi):=\pi^A(0,\xi)\!\mid\!\xi\in X^*\}$ (the group generated by the position operator $Q$)\,. 
One can also introduce $\varphi(Q):=\Op^A(\varphi\otimes 1)$ and $\psi(P^A):=\Op^A(1\otimes\psi)$ for $\varphi\in L^2(X)$ and $\psi\in L^2(X^*)$\,. One checks easily that $\varphi(Q)$ is the operator of multiplication by $\varphi$ while for zero magnetic field $\psi(P^{A=0})\equiv\psi(P)$ is the operator of convolution by the Fourier transform of $\psi$\,. Since $\varphi\otimes 1$ and $1\otimes\psi$ are not $L^2$-functions in both variables, one needs the results of \cite{MP1,IMP} for an easy justification of these objects. Equivalently, one can use formulas as $\psi(P^A):=\int_X\!dx \,\widehat\psi(x)U^A(x)$\,.

The next result is inspired by \cite[Prop. 2.2]{GI} and basically reduces to \cite[Prop. 2.2]{GI} for $A=0$\,. By $\S(Y)$ we denote the Schwartz space on the real finite-dimensional vector space $Y$\,.

\begin{Proposition}\label{beligan}
The $C^*$-algebra $\,\mathbb K\!\[L^2(X)\]$ of compact operators in $L^2(X)$ coincides with the closed vector space $\mathfrak C$ generated in $\,\mathbb B\!\[L^2(X)\]$ by products $\varphi(Q)\psi(P^A)$ with $\varphi\in \S(X)$ and $\psi\in \S(X^*)$\,.
\end{Proposition}

\begin{proof}
It is easy to check that $\varphi(Q)\psi(P^A)$ is an integral operator with kernel given for $x,y\in X$ by
\begin{equation}\label{aceea}
k^A_{\varphi,\psi}(x,y)=e^{-i\int_{[x,y]}A}\,\varphi(x)\widehat\psi(y-x).
\end{equation}
We assumed the components of $A$ to be $C^\infty_{{\rm pol}}$-functions and this immediatly implies that the magnetic phase factor in (\ref{aceea}) belongs to $C^\infty_{{\rm pol}}(X\times X)$\,. Therefore, if $\varphi\in\S(X)$ and $\psi\in\S(X)$\,, then $k^A_{\varphi,\psi}\in\S(X\times X)\subset L^2(X\times X)$ and thus $\varphi(Q)\psi(P^A)$
is a Hilbert-Schmidt operator. From this follows $\,\mathbb K\!\[L^2(X)\]\supset\mathfrak C$\,.

\medskip
Reciprocally, it is enough to show that $\mathfrak C$ contains all the integral operators with kernel
$k\in L^2(X\times X)$ (they are the Hilbert-Schmidt operators and form a dense set in $\mathbb K\!\[L^2(X)\]$)\,. Pick inside the Schwartz space $\mathcal S(X)$ an orthonormal base $\{e_i\!\mid\!i\in\N\}$
for $L^2(X)$\,. Setting
$$
F^A_{ij}(x,y):=e^{-i\int_{[x,y]}A}\,e_i(x)e_j(y-x)\,,\quad\ \forall\,x,y\in X,\ i,j\in\N\,,
$$
we get an orthonormal base $\{F^A_{ij}\mid i,j\in\N\}$ of $L^2(X\times X)$\,. So $k=\sum_{i,j}c_{ij}F^A_{ij}$, where $\sum_{i,j}|c_{ij}|^2<\infty$ and
the sum is convergent in $L^2(X\times X)$. Then the integral operator with kernel $k$ coincides with
$\sum_{i,j}c_{ij}e_i(Q)\widehat e_j(P^A)$\,. The sum converges in $\mathbb B_2\!\[L^2(X)\]$, thus in $\mathbb B\!\[L^2(X)\]$, therefore
the operator belongs to $\mathfrak C$\,.
\end{proof}

We can also state:

\begin{Theorem}\label{arsinel}
Let $\,\Omega$ a bounded subset of $\,\H:=L^2(X)$\,. The following statements are equivalent:

\begin{enumerate}
\item
The set $\,\Omega$ is relatively compact.
\item
For some (any) window $w\in\H$\,, the family $\,\phi^A_w(\Omega)$ is $\mathcal K(\Si)$-tight in $L^2(\Si)$\,.
\item
For every $\,\epsilon>0$ there exist $f\in C_{{\rm c}}(\Si)$ with $\,\underset{u\in\Omega}{\sup}\,\big\Vert\!\[\Op^A(\widehat f)-1\]\!u\,\big\Vert\,\le\epsilon$\,.
\item
One has
\begin{equation}\label{calboreanu}
\lim_{(x,\xi)\rightarrow 0}\,\sup_{u\in\Omega}\parallel\!\left[\pi^A(x,\xi)-1\right]u\!\parallel\,=\,0\,.
\end{equation}
\item
One has
\begin{equation}\label{stela}
\lim_{x\rightarrow 0}\,\sup_{u\in\Omega}\parallel\!\left[U^A(x)-1\right]u\!\parallel\,=\,0\quad{\rm and}\quad\lim_{\xi\rightarrow 0}\,\sup_{u\in\Omega}\parallel\![V(\xi)-1]\,u\!\parallel\,=\,0\,.
\end{equation}
\item
For every $\,\epsilon>0$ there exist $\varphi\in \S(X)$ and $\psi\in \S(X^*)$ with
\begin{equation}\label{epopoescu2}
\sup_{u\in\Omega}\(\,\parallel\![\varphi(Q)-1] u\!\parallel+\parallel\!\[\psi(P^A)-1\]u\!\parallel\,\)\le\epsilon\,.
\end{equation}

\end{enumerate}
\end{Theorem}

\begin{proof}
 $1 \Leftrightarrow 2 \Leftrightarrow 3$ follow from Theorem \ref{rebengiuk} by particularization, while $4\Leftrightarrow 5\,$ is trivial, taking into account the relathionships between $U^A, V$ and $\pi^A$\,. The implication $3 \Rightarrow 4 $ also holds, taking $s_0=0$ in Theorem \ref{rebengiuk} (and replacing $s$ by $-s$)\,.
 A careful examination of (\ref{tzetze}) and (\ref{fluture}) would even lead to $3 \Leftrightarrow 4 $\,, restauring the relevant convergence for arbitrary $s_0:=(x_0,\xi_0)$\,,
 but this will not be needed. $1\Rightarrow 5$ follows trivially, because $\Omega$ can be approximated by finite sets and $U^A,V$ are strongly continuous at the origin.
 
  \medskip
$5 \Rightarrow 6$ can be obtained along the same lines as the proof of the implication $4\Rightarrow 5$ in Theorem \ref{rebengiuk}, taking also into account the relations  $\psi(P^A)=\int_X\!dx \,\widehat\psi(x)U^A(x)$ and $\varphi(Q)=\int_{X^*}\!d\xi \,\widehat\varphi(\xi)V(\xi)$\,.
 
 \medskip
We finally show $6 \Rightarrow 3$\,. Let us set $T^\perp:=1-T$ and compute 
\begin{align*}
\p\!u-\varphi(Q)\psi(P^A)u\!\p&=\,\p\!\varphi(Q)\psi(P^A)^\perp u+\varphi(Q)^\perp u \!\p\\
 &\leq\,\p\!\varphi\!\p_\infty\p\!\psi(P^A)^\perp u\!\p+\p\!\varphi(Q)^\perp u \!\p\,.
\end{align*}
By using the assumption 6\,, this can be made arbitrary small uniformly in $u \in \Omega$ if $\varphi,\psi$ are chosen suitably.
As in the proof of Proposition \ref{beligan} one sees that $\varphi(Q)\psi(P^A)$ is a Hilbert-Schmidt operator. It can be approximated arbitrarily in norm by some operator $\Op^A(\widehat f)$ with $f\in C_{{\rm c}}(\Si)$ and then 3 follows easily because $\Omega$ is bounded.
\end{proof}

\begin{Remark}\label{dracu}
{\rm Many small variations are allowed in the results above. The Schwartz spaces $\S(X)$ and $\S(X^*)$ in Proposition \ref{beligan} or at point 6 of Theorem \ref{arsinel} can be replaced by other convenient "small" spaces. In Theorem \ref{arsinel}, at point 3 one could use $\Op^A(a)$ with $a\in\S(\Si)$ or with $a\in C_{{\rm c}}^\infty(\Si)$\,.
}
\end{Remark}

\begin{Remark}\label{anghelescu}
{\rm Compact operators and relative compact families of compact operators can also be treated easily in the magnetic setting, essentially combining the results of sections \ref{onecu} and \ref{one}.
}
\end{Remark}

\begin{Remark}\label{Franz}
{\rm This is connected to Remark \ref{siedentop}. In \cite[Sect. 5]{GI} one can find improvements of a classical result of Ruelle, characterizing the pure point space $\H_{{\rm p}}(H)$ and the continuous space $\H_{{\rm c}}(H):=\[\H_{{\rm p}}(H)\]^\perp$ of a self-adjoint operator $H$ acting in $L^2(X)$\,. This involves operators $\varphi(Q)$ (multiplication by $\varphi$) and $\psi(P)$ (convolution by the Fourier transform of $\psi$) that are obtained by setting $A=0$\,. The analogous magnetic results are also easily available, as corollaries of Theorem \ref{arsinel}, and they seem to be new and physically significant. A detailed discussion would need too many preparations, so we do not include it here; the interested readers would easily find the statements and the proofs by themselves.
}
\end{Remark}


\end{document}